\documentclass{amsart}
\usepackage[utf8]{inputenc}
\usepackage{amsfonts}
\usepackage{amsmath}
\usepackage{hyperref}
\usepackage{amsthm}
\usepackage{float}
\usepackage{graphicx}
\usepackage{subfigure}
\usepackage{witharrows}
\usepackage{mathtools}
\usepackage{algorithm}
\usepackage{algpseudocode}

\usepackage{comment}
\usepackage{geometry}
\newtheorem{thm}{Theorem}
 
\newtheorem{ex}[thm]{Example}

\newtheorem{lem}[thm]{Lemma}

\newtheorem*{thmLL}{Theorem (Li--Lu)}

\algrenewcommand\algorithmicrequire{\textbf{Input:}}
\algrenewcommand\algorithmicensure{\textbf{Output:}}

\title[Lift-free approaches to random rotation number]{Lift-free approaches to random rotation number \\ and numerical approximation}
\author{Zixu Li}
\address{School of Mathematics, Sun Yat-sen University, Zhuhai, P.R.China}
\email{zixu.li915@gmail.com}
\thanks{This work was partially supported by fund PGRS2006028.}
\author{Simon Lloyd}
\address{School of Mathematics and Physics, Xi'an Jiaotong-Liverpool University, Suzhou, P.R.China}
\email{Simon.Lloyd@xjtlu.edu.cn}

\begin{document}

\begin{abstract}
We study the random rotation number for random circle homeomorphisms. We introduce two new definitions of the random rotation number that can be stated without reference to any choice of lift of the dynamics to the real line, and prove that they are equivalent to the standard random rotation number. We then prove that the mean random rotation number may be approximated within an error of $1/n$ when using $n$ iterations of the dynamics. Finally, we develop numerical algorithms for approximation of the random rotation number which we test with several examples.
\end{abstract}

\maketitle



\section{Introduction}
The rotation number is a fundamental concept in the study of the dynamics of homeomorphisms of the circle. Introduced by Poincar\'e \cite{Poincare1885}
in the context of celestial mechanics, it measures the average rotation angle along orbits, and is used in the characterisation of phase-locking phenomena and the topological classification of circle homeomorphisms.

In this article we consider the generalisation of the rotation theory to random circle homeomorphisms, which consist of two parts:
\begin{enumerate}
\item ({\it Base dynamics}) a discrete-time measure-preserving dynamical system $\sigma:\Omega\to\Omega$ on a probability space, which serves as a model for a stationary noise process;
\item ({\it Fibre dynamics}) orientation-preserving circle homeomorphisms $f_\omega:S^1\to S^1$ that depend measurably on the noise.
\end{enumerate}
We aim to investigate the asymptotic properties of the dynamics on the circle of the random compositions of the form
\begin{align}
f^{(n)}_\omega:= f_{\sigma^{n-1}\omega}\circ \cdots \circ f_{\sigma\omega}\circ f_\omega 
\end{align}
for $n\in\mathbb{N}$ and typical $\omega\in\Omega$. The concept of rotation number was generalised to random circle homeomorphisms by Ruffino \cite{Ruffino2000}, and developed by Li and Lu \cite{LiLu2008}, who proved that for any integrable lift of a random circle homeomorphism to the real line, the \emph{random rotation number} exists as an integrable function. This generalises prior results about the existence of the rotation number for quasiperiodically-forced circle homeomorphisms, such as the theorem of Herman \cite{Herman1983} in the case where the base dynamics is an irrational rotation of the circle and the dependence of the fibre maps is continuous.

The definition of the random rotation number depends on making a choice of lift of the dynamics to the real line. Unlike in the deterministic case, where the lift is unique up to an integer, for random circle homeomorphisms there is much more freedom in the choice of lift, as investigated by Rodrigues and Ruffino \cite{RodriguesRuffino2013}. This motivates the search for alternative formulations of the random rotation number that do not depend on lifting the dynamics to the real line.

As the rotation number is an important topological invariant, significant efforts have been made to find effective numerical algorithms for its approximation in the case of unforced homeomorphisms (see, for example, \cite{vanVeldhuizen1988,Bruin1992,Pavani1995,SearaVillanueva2006,AlsedaBorroscullell2021}) and for quasiperiodically-forced homeomorphisms (see, for example, \cite{Starketal2002, Polotzeketal2017, DasYorke2018}). However, to the authors' knowledge, there is no prior work on methods for numerical approximation of the random rotation number.

\medskip

In this article, we first review the key properties of the deterministic rotation number and the random rotation number. The main theoretic results are two lift-free approaches to defining the random rotation number, using a binary coding method (Theorem \ref{thm:binarycoding}) and a visit-counting method (Theorem \ref{thm:visitcounting}) and a rigorous bound for approximation of the mean random rotation number (Theorem \ref{thm:1N}). In Section \ref{sec:NA}, we present numerical algorithms based on these methods for approximating the random rotation number, and in Section \ref{Sec:numerics}, we test the algorithms on some example systems.

\section{Background and notation}

In this section, we first review the deterministic rotation number and its basic properties, and then describe how the concept is extended to random compositions of circle homeomorphisms.

\subsection{Orientation-preserving circle homeomorphisms}
Let $S^1$ denote the circle, parametrized by $x\in [0,1)$. Let $\pi:\mathbb{R}\to S^1$ denote the projection map $\pi(x)=x\:\mathrm{mod}\:1 =\{x\}=x-\lfloor x\rfloor$, where $\lfloor\cdot\rfloor$ denotes the floor function and $\{x\}$ denotes the fractional part of $x$. A homeomorphism $f:S^1\to S^1$ is \emph{orientation-preserving} if there exists an increasing homeomorphism $F:\mathbb{R}\to\mathbb{R}$ such that $\pi\circ F=f\circ \pi$. Such a map $F$ is a \emph{lift} of $f$ and satisfies the \emph{degree one property}
\begin{align}\label{eq:degreeone}
F(x+1)=F(x)+1\quad \textrm{for all}\ x\in \mathbb{R}.
\end{align}
The lift is unique up to an integer: if $F_1$ and $F_2$ are both lifts of an orientation-preserving circle homeomorphism $f$, then there exists an integer $k\in\mathbb{Z}$ such that $F_2(x)=F_1(x)+k$ for all $x\in\mathbb{R}$. We denote by $\mathcal{H}$ the collection of orientation-preserving circle homeomorphisms, and by $\tilde{\mathcal{H}}$ the collection of lifts, the increasing homeomorphisms of the real line that satisfy the degree one property. 

We now recall the definition and basic properties of the rotation number - see, for example, Katok--Hasselblatt \cite{KatokHasselblatt1995} for an overview of the classical theory. The rotation number of an orientation-preserving circle homeomorphism $f\in\mathcal{H}$ can be defined as follows. For a lift $F$ of $f$ and $x\in\mathbb{R}$, the rotation number of $F$ is defined to be the limit
\begin{align}
\rho(F)=\lim_{n\to\infty} \frac{F^n(x)-x}{n}.
\end{align}
Poincar\'e \cite{Poincare1885} showed that the limit exists and is independent of $x\in\mathbb{R}$. Moreover, the bound
\begin{align}\label{eq:rhodesthompson}
\left| \frac{F^n(0)}{n} -\rho(F) \right|\leq \frac{1}{n}
\end{align}
(see Rhodes and Thompson \cite{RhodesThompson1991}) ensures control of the error when approximating $\rho(F)$. The fractional part of the rotation number $\rho(F)$ is independent of the choice of lift, and so the rotation number of $f$ can be defined as $\rho(f)=\pi(\rho(F))$ for any lift $F$ of $f$. 

The rotation number has several important properties. It is invariant under topological conjugacy, or even the weaker notion of topological semiconjugacy: that is, given $f,g\in\mathcal{H}$ and continuous monotone map $h:S^1\to S^1$ (not necessarily a homeomorphism) such that $h\circ f=g\circ h$, then $\rho(f)=\rho(g)$. The rotation number varies monotonically: if $F,G\in\tilde{\mathcal{H}}$ with $F(x)<G(x)$ for all $x\in\mathbb{R}$, then $\rho(F)\leq \rho(G)$. Moreover, the inequality is strict if either $\rho(F)$ or $\rho(G)$ is irrational.
Rationality of the rotation number characterises the existence of periodic orbits: the rotation number is rational if and only if there is a periodic point. Moreover, if the rotation number is a rational number $p/q$ (in lowest terms), then there exists a periodic orbit of period $q$ and the points of the orbit are permuted by $f$ in the same way as the points of the orbit of $0$ are permuted by the rotation by $p/q$.
Finally, for $k\in\mathbb{N}$, there is a formula for the rotation number of the $k$th iterate of $F$ in terms of $\rho(F)$:
\begin{align}\label{eq:rhoiteration}
\rho(F^k)=k\rho(F).
\end{align}

A lift $F\in\tilde{\mathcal{H}}$ can be studied in terms of the \emph{displacement function} $\Delta_F:\mathbb{R}\to\mathbb{R}$ given by $\Delta_F(x)=F(x)-x$, which quantifies the deviation of $F$ from the identity map. Since $\Delta_F$ is continuous and $1$-periodic by (\ref{eq:degreeone}), we can equip $\tilde{\mathcal{H}}$ with a norm based on the supremum norm of the displacement functions:
\begin{align}
\|F\|_{\tilde{\mathcal{H}}}=\|\Delta_F\|=\sup_{x\in\mathbb{R}}|\Delta_F(x)|.
\end{align}
A metric on $\mathcal{H}$ is then given by
\begin{align}
d_{\mathcal{H}}(f,g)=\inf\left\{\|F-G\|_{\tilde{\mathcal{H}}}: F\textrm{ is a lift of }f, G\textrm{ is a lift of }g \right\}.
\end{align}
The rotation number $\rho$ is continuous with respect to the resulting topologies on $\tilde{\mathcal{H}}$ and on $\mathcal{H}$.

\subsection{Random circle homeomorphisms}

We now formalise the concept of random circle homeomorphism to fix notation. 
Recall that a \emph{measure-preserving dynamical system} is a quadruple ($\Omega,\mathcal{F},\mathbb{P},\sigma)$, where $\Omega$ is a set, $\mathcal{F}$ is a $\sigma$-algebra of subsets, $\mathbb{P}$ is a probability measure and $\sigma:\Omega\to\Omega$ is a measurable transformation that preserves the probability measure $\mathbb{P}$: that is $\mathbb{P}(\sigma^{-1}(A))=\mathbb{P}(A)$ for each set $A\in\mathcal{F}$. We write $\sigma\omega$ for $\sigma(\omega)$.

A \emph{random circle homeomorphism (RCH)} over $(\Omega,\mathcal{F},\mathbb{P},\sigma)$ is a measurable map $\phi:\mathbb{N}_0\times\Omega\times S^1\to S^1$ with the following properties for each $n,m\in \mathbb{N}_0=\{0,1,2,\ldots\}$, $\omega\in\Omega$ and $x\in S^1$:
\begin{enumerate}
\item $\phi(0,\omega,\cdot)=\mathrm{Id}_{S^1}$;
\item $\phi(n,\omega,\cdot)\in \mathcal{H}$;
\item $\phi(m+n,\omega,x)=\phi(n,\sigma^m\omega,\phi(m,\omega,x))$.
\end{enumerate}
The measure-preserving dynamical system $(\Omega,\mathcal{F},\mathbb{P},\sigma)$ is called the \emph{base dynamics} of $\phi$. If we denote $\phi(n,\omega,\cdot)$ by $f^{(n)}_\omega$ and $\phi(1,\omega,\cdot)$ by $f_\omega$, then
\begin{align}\label{eq:fnomega}
\phi(n,\omega,\cdot)=f^{(n)}_\omega=f_{\sigma^{n-1}\omega}\circ \cdots \circ f_{\sigma\omega}\circ f_{\omega}.
\end{align}

Let $\mathcal{H}(\Omega)$ denote the collection of measurable maps $f:\Omega\to \mathcal{H}$, where $\omega\mapsto f_\omega$. We say $f\in \mathcal{H}(\Omega)$ is the \emph{generator} of the RCH $\phi:\mathbb{N}_0\times\Omega\times S^1\to S^1$ over $(\Omega,\mathcal{F},\mathbb{P},\sigma)$ defined by (\ref{eq:fnomega}). For brevity, we denote this RCH simply by $(\sigma,f)$. 

We extend the concept of lift to random circle homeomorphisms: a measurable map $F:\Omega\to\tilde{\mathcal{H}}$ is said to be a \emph{lift} of $f\in \mathcal{H}(\Omega)$ if $\pi\circ F_\omega=f_\omega\circ \pi$ for each $\omega\in\Omega$. We let $\tilde{H}(\Omega)$ denote the collection of measurable maps $F:\Omega\to\tilde{\mathcal{H}}$ that satisfy the integrability condition
\begin{align}
\int_{\omega\in\Omega} \|F_\omega\|_{\tilde{\mathcal{H}}}\:\mathrm{d}\mathbb{P}(\omega)
=\int_{\omega\in\Omega} \|\Delta_{F_\omega}\|\:\mathrm{d}\mathbb{P}(\omega)<\infty. 
\end{align}

Random circle homeomorphism are a generalisation of other classes of dynamical system. If $\Omega=S^1$ is the circle, the base transformation $\sigma:\Omega\to\Omega$ is an irrational rotation and the map $f:\Omega\to\mathcal{H}$ is continuous, then $(\sigma,f)$ is a \emph{quasiperiodically-forced circle homeomorphism}, a class of systems of much interest due to the existence of strange non-chaotic attractors (\cite{Herman1983}, \cite{Grebogietal1984}).
We can recover the deterministic dynamics of a single circle homeomorphism by taking $\Omega$ to be a singleton set $\{p\}$ and $\sigma$ to be the identity transformation.

The concept of rotation number was generalised to random circle homeomorphisms by Ruffino \cite{Ruffino2000} for lifts satisfying $F_\omega\in (-1/2,1/2]$ for all $\omega\in\Omega$. Later  Li and Lu \cite{LiLu2008} proved that for any integrable lift $F\in\tilde{\mathcal{H}}(\Omega)$  of a random circle homeomorphism, the \emph{random rotation number} exists as an integrable function $\rho_F:\Omega\to\mathbb{R}$.

\begin{thmLL}
If $F\in \tilde{\mathcal{H}}(\Omega)$, then there exists $\rho_F\in L^1(\Omega)$ such that for each $x\in\mathbb{R}$
\begin{align}\label{eqn:randomrotationnumber}
\lim_{n\to\infty} \frac{F^{(n)}_\omega(x)-x}{n} = \rho_F(\omega)
\end{align}
for $\mathbb{P}$-almost every $\omega\in\Omega$. The function $\rho_F$ satisfies $\rho_F\circ \sigma=\rho_F$, and thus is essentially constant if $\sigma$ is ergodic.
Moreover, $\rho: \tilde{\mathcal{H}}(\Omega)\to L^1(\Omega)$ is continuous.

If, in addition, the base dynamics $\sigma:\Omega\to\Omega$ is a continuous and uniquely ergodic transformation of a compact metric space, then the function $\rho_F$ is (everywhere) constant. 
\end{thmLL}

A well-studied family of circle homeomorphisms was introduced by Arnold \cite{Arnold1957}, which arises in the study of the periodically-forced pendulum.
Given measurable functions $\alpha:\Omega\to[-1,1]$ and $\beta:\Omega\to \mathbb{R}$, we can form a \emph{random Arnold homeomorphism} by setting
\begin{align}\label{eq:Arnold}
f_\omega(x) = x+\frac{\alpha(\omega)}{2\pi}\sin(2\pi x) + \beta(\omega) \:\mathrm{mod}\:1
\end{align}
for all $x\in S^1$ and $\omega\in\Omega$. The requirement that $|\alpha(\omega)|\leq 1$ ensures that each lift of $f_\omega$ is strictly increasing, and so $f\in \mathcal{H}(\Omega)$. 

\subsection{Skew product formulation and acceleration}

An equivalent way of formulating a random circle homeomorphism is through a the concept of a skew product.
Consider a RCH $\phi:\mathbb{N}_0\times \Omega\times S^1\to S^1$ over the base dynamics $(\Omega,\mathcal{F},\mathbb{P},\sigma)$ with generator $f\in\mathcal{H}(\Omega)$. The \emph{associated skew product} is the map $\Phi:\Omega\times S^1\to\Omega\times S^1$ defined by
\begin{align*}
\Phi(\omega,x)=(\sigma\omega, f_\omega(x)),
\end{align*}
and so for $n\in\mathbb{N}_0$,
\begin{align*}
\Phi^n(\omega,x)=(\sigma^n\omega, f^{(n)}_\omega(x))= (\sigma^n\omega, \phi(n,\omega,x)).
\end{align*}

Let $\mathcal{P}_{\mathbb{P}}$ denote the set of probability measures $\mu$ on $\Omega\times S^1$ with marginal $\mathbb{P}$ on $\Omega$: that is for which ${\pi_\Omega}_\star \mu= \mathbb{P}$, where $\pi_\Omega$ is the natural projection onto $\Omega$. Let $\mathcal{I}(\Phi)$ denote the subset of $\mathcal{P}_{\mathbb{P}}$ of those probability measures that are invariant for the skew product $\Phi$. Since $S^1$ is a compact metric space and the fibre maps are continuous, the subset $\mathcal{I}(\Phi)$ is non-empty (see Arnold \cite{Arnold1998}). As shown by Ruffino--Rodrigues \cite{RodriguesRuffino2013}, for any $\mu\in\mathcal{I}(\Phi)$, the $\mathbb{P}$-mean value $\hat{\rho}(F)$ of the random rotation number of $F\in\tilde{\mathcal{H}}(\Omega)$ satisfies
\begin{align}\label{eqn:skewint}
\hat{\rho}(F)=\int_\Omega \rho_F\:\textrm{d}\mathbb{P} = \int_{\Omega\times S^1}\delta_F\:\textrm{d}\mu,
\end{align}
where $\delta_F:\Omega\times S^1\to\mathbb{R}$ is given by 
\begin{align}\label{eqn:smalldelta}
\delta_F(\omega,x)=\Delta_{F_\omega}(\pi^{-1}(x)).
\end{align}

Consider a measure-preserving dynamical system $(\Omega,\mathcal{F},\mathbb{P},\sigma)$ and let $k\in\mathbb{N}$. Then $\tau=\sigma^k:\Omega\to\Omega$, the $k$th iterate of $\sigma$ also preserves the probability measure $\mathbb{P}$ and so $(\Omega,\mathcal{F},\mathbb{P},\tau)$ is a measure-preserving dynamical system. Given $f\in\mathcal{H}(\Omega)$, the map $g=f^{(k)}\in\mathcal{H}(\Omega)$ given by
\begin{align}
g_\omega(x)=f^{(k)}_\omega(x)
\end{align}
for each $\omega\in\Omega$ and $x\in\mathbb{R}$ is called the \emph{$k$-fold acceleration} of $f$. We extend the definition to the RCH itself by saying the $k$-fold acceleration of the RCH $(\sigma,f)$ is $(\tau,g)=(\sigma^k,f^{(k)})$. The $k$-fold acceleration of a lift of a RCH can be defined in similar way.

For $k\in\mathbb{N}$, let $G=F^{(k)}\in\tilde{\mathcal{H}}$ be the $k$-fold acceleration of $F\in\tilde{\mathcal{H}}$. Then for each $n\in\mathbb{N}$,
\begin{align}
\frac{1}{n}G^{(n)}_\omega(x) = \frac{1}{n}F^{(nk)}_\omega(x)= k\, \frac{1}{nk}F^{(nk)}_\omega(x).
\end{align}
Hence, by applying Li and Lu's theorem, we have an analogous result to equation (\ref{eq:rhoiteration}):
\begin{align}
\rho_{F^{(k)}}(\omega)=k\, \rho_F(\omega) 
\end{align}
for $\mathbb{P}$-almost every $\omega\in\Omega$.

\section{Lift-free approaches to random rotation number}\label{sec:liftfree}

In this section we develop two methods for calculating the random rotation number of a random circle homeomorphism that do not involve taking lifts.

\subsection{Lift dependence of the random rotation number}
For a random circle homeomorphism $(\sigma,f)$, there are different ways that the dynamics can be lifted to the real line.
Rodrigues and Ruffino \cite{RodriguesRuffino2013} define  the \emph{$(q,\alpha)$-lift} $F\in\hat{\mathcal{H}}(\Omega)$ for each pair $(q,\alpha)\in\mathbb{R}^2$, where $F_\omega$ is the unique lift of $f_\omega$ that satisfies $F_\omega(q)\in [\alpha,\alpha+1)$ for each $\omega\in \Omega$. 
In particular, the \emph{standard lift} is the $(0,0)$-lift: that is, the unique lift that satisfies
\begin{align}\label{eq:standardlift}
F_\omega(0)\in [0,1)
\end{align}
for each $\omega\in \Omega$. 

The choice of lift affects the random rotation number, as can be demonstrated by the following example in the context of quasiperiodically-forced rotations.

\begin{ex}
Let $\Omega=S^1$ and let $\sigma:\Omega\to \Omega$ be the irrational rotation given by
$\sigma\omega=\omega+\xi\:\mathrm{mod}\:1$, for some irrational number $\xi\in\mathbb{R}\backslash\mathbb{Q}$. Let $f\in\mathcal{H}(\Omega)$ be the random rotation defined by $f_\omega(x)=x+\beta(\omega)\:\mathrm{mod}\:1$, where $\beta:\Omega\to\mathbb{R}$ is defined by
\begin{align}
\beta(\omega)=\begin{cases}
4\omega\quad & \textrm{for}\quad 0\leq \omega<1/2, \\
4-4\omega\quad & \textrm{for}\quad 1/2\leq \omega<1.
\end{cases}
\end{align}
We can consider various lifts of $f$.
For $k\in\mathbb{Z}$, let $F_1,F_2,F_3\in \hat{\mathcal{H}}(\Omega)$ be defined by
\begin{align*}
F_{1,\omega}(x) & =x+\beta(\omega)+k \\
F_{2,\omega}(x) & =x+\{\beta(\omega)\}+k \\
F_{3,\omega}(x) & =x+\left\{\beta(\omega)+\frac12\right\}-\frac12,
\end{align*}
where $\{x\}=x-\lfloor x\rfloor$ denotes the fractional part of $x$. These choices of lift each have their own advantages: $F_1$ exhibits continuous dependence on $\omega$, $F_2$ is the $(0,k)$-lift and so satisfies $F_{2,\omega}(0)\in [k,k+1)$ for all $\omega\in\Omega$, and $F_3$ has zero mean displacement (as demonstrated below).
Let $\mu$ denote the Lebesgue measure on $\Omega\times S^1$. Since $\sigma$ and each $f_\omega$ are isometries, $\mu$ is invariant for the associated skew product.
For each $i\in\{1,2,3\}$, we have that the displacement function $\Delta_{F_i}(\omega,x)=F_{i,\omega}(x)-x$ is independent of $x$, and so equation (\ref{eqn:skewint}) can be used to calculate the mean random rotation number:
\begin{align*}
\hat{\rho}(F_1) & = \int_{\Omega\times S^1} \delta_{F_1}\:\mathrm{d}\mu=\int_{\Omega} \beta\:\mathrm{d}\mathbb{P}+k=k+1, \\
\hat{\rho}(F_2) & = \int_{\Omega\times S^1} \delta_{F_2}\:\mathrm{d}\mu=\int_{\Omega} \{\beta\}\:\mathrm{d}\mathbb{P}+k=k+\frac{1}{2}, \\
\hat{\rho}(F_3) & = \int_{\Omega\times S^1} \delta_{F_3}\:\mathrm{d}\mu=\int_{\Omega} \left\{\beta+\frac12\right\}\:\mathrm{d}\mathbb{P}-\frac{1}{2}=0.
\end{align*}
\end{ex}

\subsection{Binary coding}
In this section we develop an orbit coding technique for random circle homeomorphisms, building on a method used in the deterministic case (see de Melo and van Strien \cite{deMelovanStrien1993}). 

Given a circle homeomorphism $f\in\mathcal{H}$, we let $\hat{f}:[0,1)\to [0,1)$ denote the equivalent interval map given by forgetting the identification of the endpoints of the interval. Let $c_f=f^{-1}(0)\in[0,1)$ and define the \emph{left interval} $I=[0,c_f)$ and the \emph{right interval} $J=[c_f,1)$.
If $f(0)\neq 0$, then the map $\hat{f}$ has a single discontinuity point, at $c_f\in(0,1)$, and two continuous and strictly increasing branches, supported on $I$ and $J$. If $f(0)=0$, then $c_f=0$ and map $\hat{f}$ is continuous and so has a single branch $I=[0,1)$, and $J=\emptyset$.
The orbit of a point $x_0\in[0,1)$ can be tracked by recording whether each iterate lies in the left or right interval, labeling visits to $I$ by $0$ and visits to $J$ by $1$. The resulting binary sequence encodes the dynamics of the orbit. 

We follow a similar approach in the random setting. Given $f\in \mathcal{H}(\Omega)$, let 
$\hat{f_\omega}$ denote the equivalent interval map to $f_\omega$. We let $c_\omega={f_\omega}^{-1}(0)$, so $c_\omega$ is either the unique discontinuity point of $\hat{f_\omega}$ or else $c_\omega=0$ if $\hat{f_\omega}$ is continuous. Let
\begin{align*}
J_\omega=\begin{cases}
[c_\omega,1) & \textrm{if}\ \hat{f}_\omega(0)\neq 0, \\
\emptyset & \textrm{if}\ \hat{f}_\omega(0)= 0,
\end{cases}
\end{align*}
denote the right interval for $\hat{f}_\omega$ and let $I_\omega=[0,1)\backslash J_\omega$ denote the left interval.
Let $S_\omega:[0,1)\to\mathbb{R}$ denote the indicator function of the interval $J_\omega$: that is $S_\omega(x)=\chi_{J_\omega}(x)$ for $x\in [0,1)$.

Given $x_0\in [0, 1)$ and $\omega\in\Omega$ and $n\in \mathbb{N}$, we define the $n$th \emph{binary coding frequency}
\begin{align}\label{eq:binary}
B_f(n,\omega,x_0)=\frac{1}{n}\sum_{k=0}^{n-1} S_{\sigma^{k}\omega}(\hat{f}^{(k)}_\omega(x_0)),
\end{align}
which measures the frequency with which the orbit of $x_0$ visits the right interval. The following theorem shows that the limiting average of the binary coding frequency exists almost surely and is equal to the random rotation number of the standard lift.

\begin{thm}\label{thm:binarycoding}
Let $(\Omega, \mathcal{F}, \mathbb{P}, \sigma)$ be a measure-preserving dynamical system. Let $f\in \mathcal{H}(\Omega)$ and let $F\in\tilde{\mathcal{H}}$ be the standard lift of $f$. Then, for each $x_0\in [0,1)$ and $n\in\mathbb{N}$, 
\begin{align}\label{eq:bitcountingerror}
\left|\frac{1}{n}F^{(n)}_{\omega}(x_0)-B_f(n,\omega,x_0)\right|< \frac{1}{n}
\end{align}
and for $\mathbb{P}$-almost every $\omega\in\Omega$,
\begin{align}
\lim_{n\rightarrow \infty}B_f(n,\omega,x_0)= \rho_F(\omega).
\end{align}
\end{thm}
\begin{proof}
For each $x\in \mathbb{R}$ and $\omega\in \Omega$, the standard lift satisfies $F_{\omega}(x)=\hat{f}_\omega(\pi(x))+S_{\omega}(\pi(x))+\lfloor x\rfloor$ where $\hat{f}$ and $S_\omega$ are defined as before. So for each $n\in \mathbb{N}$, by induction,
\begin{align*}
F^{(n)}_{\omega}(x)&=\hat{f}^{(n)}_\omega(\pi(x))+\sum_{k=0}^{n-1}S_{\sigma^k\omega}(\hat{f}^{(k)}_\omega(\pi(x)))+\lfloor x \rfloor\,.
\end{align*}
Therefore
\begin{align*}
\frac{1}{n}F^{(n)}_{\omega}(x) =
\frac{1}{n}\hat{f}^{(n)}_\omega(\pi(x))+B_f(n,\omega,\pi(x))+\frac{1}{n}\lfloor x \rfloor.
\end{align*}
For $x_0\in [0, 1)$, since $0 \leq \hat{f}^{(n)}_{\omega}(x_0)<1$ for all $\omega\in \Omega$ and $n\in \mathbb{N}$, we have
\begin{align}
\left|\frac{1}{n}F^{(n)}_{\omega}(x_0)-B_f(n,\omega,x_0)\right| =\frac{\hat{f}^{(n)}_{\omega}(x_0)}{n}< \frac{1}{n}.
\end{align}
Taking the limit, by (\ref{eqn:randomrotationnumber}) we have
\begin{align}
\lim_{n\rightarrow \infty}B_f(n,\omega,x_0)=\rho_F(\omega)
\end{align}
for $\mathbb{P}$-almost every $\omega\in\Omega$.
\end{proof}

\subsection{Visit counting}

Building on an idea used by Guckenheimer and Holmes \cite{GuckenheimerHolmes1983} in the deterministic setting, another approach is to count the number of visits to a fundamental domain. Given a homeomorphism $f\in \mathcal{H}$ and a reference point $z\in S^1$ that is not fixed by $f$, consider the fundamental domain interval $[z,f(z))\subset S^1$. (Note that if $0\leq f(z)<z<1$, then this notation refers to the connected set $[z,1)\cup[0,f(z))$.) Each time that an orbit completes one lap of the circle, it visits the fundamental domain precisely once, and so the rotation number can be calculated in terms of the asymptotic frequency of visits to the fundamental domain.

We can take a similar approach in the random setting. Given $f\in \mathcal{H}(\Omega)$ and $x_0,z\in S^1$, for each $n\in \mathbb{N}$ and $\omega\in\Omega$ we define the $n$th \emph{visit frequency} to be
\begin{align}\label{eq:visit}
V_f(n,\omega,x_0,z) = \frac{1}{n}\# \{1\leq i\leq n: f^{(i)}_\omega(x_0)\in [z,f_{\sigma^{i-1}\omega}(z))\}.
\end{align}
In the case $a=b$, the interval $[a,b)\subset S^1$ is defined to be the empty set.

\begin{thm}\label{thm:visitcounting}
Let $(\Omega, \mathcal{F}, \mathbb{P}, \sigma)$ be a measure-preserving dynamical system. Let $f\in \mathcal{H}(\Omega)$ and let $F\in\tilde{\mathcal{H}}$ be the standard lift of $f$. Then, for each $x_0\in S^1$ and $n\in\mathbb{N}$, 
\begin{align}
V_f(n,\omega,x_0,0)=B_f(n,\omega,x_0),
\end{align}
and for $\mathbb{P}$-almost every $\omega\in\Omega$,
\begin{align}\label{eq:visitwith0}
\lim_{n\rightarrow \infty}V_f(n,\omega,x_0,0)= \rho_F(\omega).
\end{align}
Moreover, if $f$ has no fixed points, then for any $z\in S^1$ and $\mathbb{P}$-almost every $\omega\in\Omega$,
\begin{align}\label{eq:visitwithz}
\lim_{n\rightarrow \infty}V_f(n,\omega,x_0,z)= \rho_F(\omega).
\end{align}
\end{thm}
\begin{proof}
For each $n\in\mathbb{N}$, we have
\begin{align*}
V_f(n,\omega,x_0,0) & = \#\{1\leq i \leq n: f^{(i)}_\omega(x_0)\in [0,f_{\sigma^{i-1}\omega}(0)) \}/n \\
& = \#\{0\leq i \leq n-1: f^{(i+1)}_\omega(x_0)\in [0,f_{\sigma^{i}\omega}(0)) \}/n \\
& = \#\{0\leq i \leq n-1: f^{(i)}_\omega(x_0)\in [(f_{\sigma^i\omega})^{-1}(0),1) \}/n \\
& = \#\{0\leq i \leq n-1: f^{(i)}_\omega(x_0)\in J_{\sigma^i\omega} \}/n \\
& = B_f(n,\omega,x_0).
\end{align*}
Hence the limit follows from Theorem \ref{thm:binarycoding}.

Now suppose $f$ has no fixed points and fix $z\in S^1$. Then for each $i\in\mathbb{N}$,  $J_{i,\omega,z}=[(f^{(i)}_\omega)^{-1}(z),(f^{(i-1)}_\omega)^{-1}(z))$ is a non-empty interval, and the endpoints depend continuously on $z$. We have that for all $x\in S^1$, $z\in A$, $1\leq i \leq n$ and almost every~$\omega\in \Omega$,
\begin{align} 
f^{(i)}_\omega(x)\in [z, f_{\sigma^{i-1}\omega}(z)) \quad\textrm{if and only if}\quad  x \in J_{i, \omega, z}.
\end{align}
The set of endpoints of the intervals of the collection $\{J_{i, \omega, z}\}_{i=1}^n$ consists of the points $\{z, (f_\omega)^{-1}(z), \ldots, (f^{(n)}_\omega)^{-1}(z)\}$ of the backward orbit of $z$. For each $i\in \{1, \ldots, n\}$, the left endpoint of the interval $J_{i, \omega, z}$ coincides with the right endpoint of the interval $J_{i+1, \omega, z}$. Therefore, the only possible discontinuities of the function $x\mapsto V_f(n,\omega, x, z)$ are the first and last points of the list, $z$ and $(f^{(n)}_\omega)^{-1}(z)$. Suppose that $z'\in S^1$ is close to $z$. Then the backward orbits of $z$ and $z'$ satisfy
\begin{align}
(f^{(i-1)}_\omega)^{-1}(z)\in J_{i, \omega, z'}
\end{align}
for each $1\leq i \leq n$. Likewise, for $i\in \{1,\ldots,n\}$, we have $(f^{(i)}_\omega)^{-1}(z')\in J_{i, \omega, z}$. Hence, the integers 
\begin{align*}
\#\{1 \leq i \leq n  : f^{(i)}_\omega(x_0)\in [z,f_{\sigma^{i-1}\omega}(z))\}
\  \textrm{and}\ 
\#\{1 \leq i \leq n : f^{(i)}_\omega(x_0)\in [z',f_{\sigma^{i-1}\omega}(z'))\}
\end{align*}
differ by at most 1. Thus by (\ref{eq:visitwith0}) for $\mathbb{P}$-almost every~$\omega\in\Omega$, we have 
\begin{align}
\lim_{n\rightarrow \infty}V_f(n,\omega,x_0,z)= \rho_F(\omega).
\end{align}
\end{proof}

\section{Error bounds for the mean random rotation number}\label{sec:errorbounds}

The random rotation number $\rho_F$ is defined as a limit, which raises the question of the extent to which it can be approximated by finitely many terms. 
The following example, based on the work of Godr\`eche et al.~\cite{Godrecheetal1987}, illustrates the difficulties encountered when approximating the random rotation number.

\begin{ex}\label{ex:Randomrotation}
Let $\sigma:\Omega\to\Omega$ denote the irrational rotation by $\tau=(3-\sqrt{5})/2\notin\mathbb{Q}$, which preserves the Lebesgue measure $\mathbb{P}$ on the circle $\Omega=S^1$. Let $F\in\hat{H}(\Omega)$ be given by $f_\omega(x)=x+R(\omega)$, where $R:\Omega\to\mathbb{R}$ is given by
\begin{align}
R(\omega)=\begin{cases}
+1 \quad \omega\in [0,1/2), \\
-1 \quad \omega\in [1/2,1).
\end{cases}
\end{align}
Thus $(\sigma,F)$ is the lift of a random rotation $(\sigma,f)$, and so two-dimensional Lebesgue measure $\mu$ on $\Omega\times S^1$ is invariant under the associated skew product. Since $\delta_F(\omega,x)=R(\omega)$ is independent of $x\in S^1$, by (\ref{eqn:skewint}) we have that the random rotation number $\rho_F$ is zero $\mathbb{P}$-almost everywhere.
Godreche et al.~\cite{Godrecheetal1987} show that
\begin{align}
F^{(n)}_0(0)=\sum_{i=1}^n R(\sigma^i\omega)   
\end{align}
is unbounded. More specifically, they show that for each $m\in \mathbb{N}$, we have $F^{(n)}_0(0)=m$ when $n=\sum_{k=1}^m a_{6k-4}$, where $(a_n)$ is the Fibonacci sequence defined by $a_n=a_{n-1}+a_{n-2}$, $a_0=0$ and $a_1=1$. In particular, the subsequence of `record highs' of $F^{(n)}_0(0)$ begins with $F^{(1)}_0(0)=1$, $F^{(22)}_0(0)=2$, $F^{(399)}_0(0)=3$ and $F^{(7164)}_0(0)=4$. 
Let $A_n$ denote the approximation $F^{(n)}_0(0)/n$ of the random rotation number. Figure \ref{fig:Randomrotation} shows a plot of $A_n$ against $n$, in which we can observe that for $n\leq 10000$, the points $A_n$ lie on curves of the form $c/n$, where $c$ is an integer in the set $\{-3,\ldots,4\}$.

\begin{figure}[htb]
    \centering
    \includegraphics[scale=1.00]{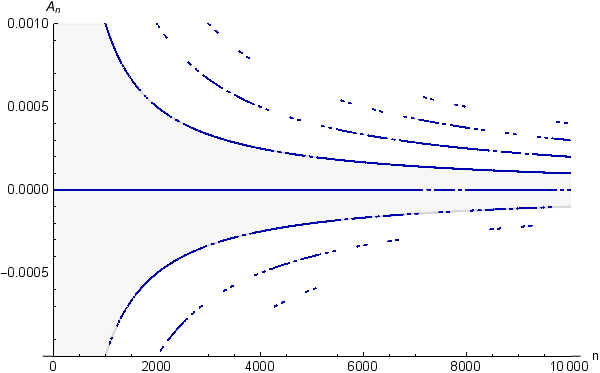}  
    \caption{Graph showing approximations $A_n$ of the random rotation number against $n$, and the region bounded by the curves $\pm1/n$ (light grey).}
    \label{fig:Randomrotation}
\end{figure}
\end{ex}

Example \ref{ex:Randomrotation} shows that, in the random case, there is no equivalent of the $1/n$ error bound provided by equation (\ref{eq:rhodesthompson}), and in fact no $C/n$ bound for any $C>0$. Therefore, it is not advisable to try to estimate the random rotation number by iterating a single trajectory.

A more effective approach is to estimate the random rotation number by averaging over all trajectories. Stark et al.~\cite{Starketal2002} prove that for quasiperiodically-forced systems where the base is an irrational rotation (thus preserving the Lebesgue measure $\mathbb{P}$) and the generator depends continuously on $\omega$, the mean rotation number $\hat{\rho}(F)$ satisfies
\begin{align}
\left|\int_\Omega \frac{F^{(n)}_\omega(x)-x}{n}\:\mathrm{d}\mathbb{P}(\omega)   -\hat{\rho}(F) \right| \leq \frac{1}{n}
\end{align}
for each $n\in\mathbb{N}$ and $x\in\mathbb{R}$.

In this section, we follow an approach similar to that of Stark et al.~\cite{Starketal2002} to prove an analogous result for random circle homeomorphisms, namely, that the mean random rotation number $\hat{\rho}(F)$ is well-approximated by $(1/n)R_n$, where
\begin{align}\label{eq:intapprox}
R_n = \int_\Omega F^{(n)}_\omega(0)\:\mathrm{d}\mathbb{P}(\omega),
\end{align}
with an error of at most $1/n$. For $n\in\mathbb{N}$, let $\delta^{(n)}_F:\Omega\times S^1\to\mathbb{R}$ denote the map $\delta^{(n)}_F(\omega,x)=\delta_{F^{(n)}}(\omega,x)=\Delta_{F^{(n)}_\omega}(\pi^{-1}(x))$. We shall need the following results about the basic properties of  $\delta_F^{(n)}$.

\begin{lem}\label{lem:hermanbound}
Let $(\Omega, \mathcal{F}, \mathbb{P}, \sigma)$ be a measure-preserving dynamical system. Let $f\in \mathcal{H}(\Omega)$ and let $F\in\tilde{\mathcal{H}}(\Omega)$ be a lift of $f$. Then for each $\omega\in\Omega$, $x,y\in S^1$ and $n\in\mathbb{N}$, we have
\begin{align}
\left|\delta_{F}^{(n)}(\omega,x)-\delta_{F}^{(n)}(\omega,y)\right| < 1.
\end{align}
\end{lem}
\begin{proof}
If $x=y$, we are done, so we assume $x\neq y$. We can take $x',y'\in\mathbb{R}$, such that $x'<y'<x'+1$ and $\pi(x')=x$ and $\pi(y')=y$. Then by the strict monotonicity of the lift, we have
\begin{align*}
\Delta_{F^{(n)}_\omega}(x')=F^{(n)}_\omega(x')-x'< F^{(n)}_\omega(y')-x'<F^{(n)}_\omega(y')-y'+1=\Delta_{F^{(n)}_\omega}(y')+1.
\end{align*}
Similarly, we have
\begin{align*}
\Delta_{F^{(n)}_\omega}(x')=\Delta_{F^{(n)}_\omega}(x'+1)=F^{(n)}_\omega(x'+1)-x'-1>F^{(n)}_\omega(y')-x'-1>F^{(n)}_\omega(y')-y'-1=\Delta_{F^{(n)}_\omega}(y')-1.
\end{align*}
Combining these inequalities, we find
\begin{align*}
\left| \Delta_{F^{(n)}_\omega}(x')-\Delta_{F^{(n)}_\omega}(y')\right|< 1,
\end{align*}
and thus, projecting to the circle, we obtain
\begin{align*}
\left| \delta_{F}^{(n)}(\omega,x)-\delta_F^{(n)}(\omega,y)\right|< 1.
\end{align*}
\end{proof}

\begin{lem}\label{lem:devbound}
Let $(\Omega, \mathcal{F}, \mathbb{P}, \sigma)$ be a measure-preserving dynamical system. Let $f\in \mathcal{H}(\Omega)$ and let $F\in\tilde{\mathcal{H}}(\Omega)$ be a lift of $f$. Then for each $\omega\in\Omega$, $x,y\in S^1$ and $n,m\in\mathbb{N}$, we have
\begin{align}
\left|\delta_{F}^{(n+m)}(\omega,x)-\delta_{F}^{(m)}(\sigma^n\omega,y)-\delta_{F}^{(n)}(\omega,x)\right| < 1.
\end{align}
\end{lem}
\begin{proof}

For $x\in S^1$, take any $z\in\mathbb{R}$ such that $\pi(z)=x$. By item 3 of the definition of a random circle homeomorphism, we have that the lift $F$ satisfies 
\begin{align*}
F^{(n+m)}_\omega(z)  = F^{(m)}_{\sigma^n\omega}(F^{(n)}_\omega(z)) 
& \implies F^{(n+m)}_\omega(z)-z  = F^{(m)}_{\sigma^n\omega}(F^{(n)}_\omega(z)) - F^{(n)}_\omega(z) +F^{(n)}_\omega(z)-z \\
& \implies \Delta_{F^{(n+m)}_\omega}(z)  = \Delta_{F^{(m)}_{\sigma^n\omega}}(F^{(n)}_\omega(z))+\Delta_{F^{(n)}_\omega}(z) \\
& \implies \delta_{F}^{(n+m)}(\omega,\pi(z))   = \delta_{F}^{(m)}(\sigma^n\omega,\pi(F^{(n)}_\omega(z)))+\delta_{F}^{(n)}(\omega,\pi(z)) \\
& \implies \delta_{F}^{(n+m)}(\omega,x)   = \delta_{F}^{(m)}(\sigma^n\omega,f^{(n)}_\omega(x))+\delta_{F}^{(n)}(\omega,x).
\end{align*}
By Lemma \ref{lem:hermanbound}, we have $|\delta_F^{(m)}(\sigma^n\omega,f^{(n)}_\omega(x))-\delta_F^{(m)}(\sigma^n\omega,y)|<1$ for any $y\in S^1$, which gives
\begin{align*}
\left|\delta_{F}^{(n+m)}(\omega,x)-\delta_{F}^{(m)}(\sigma^n\omega,y)-\delta_{F}^{(n)}(\omega,x)\right| < 1.
\end{align*}
\end{proof}

Using these lemmas, we can prove the following bound for the mean random rotation number.

\begin{thm}\label{thm:1N}
Let $(\Omega, \mathcal{F}, \mathbb{P}, \sigma)$ be a measure-preserving dynamical system. Let $f\in \mathcal{H}(\Omega)$ and let $F\in\tilde{\mathcal{H}}(\Omega)$ be a lift of $f$. Then, for each $n\in\mathbb{N}$, the mean random rotation number $\hat{\rho}(F)$ satisfies
\begin{align}\label{eq:mean1nbound}
\left| \frac{1}{n}\int_{\Omega}F^{(n)}_\omega(0)\:\mathrm{d}\mathbb{P} - \hat{\rho}(F) \right| \leq \frac{1}{n}.
\end{align}
\end{thm}
\begin{proof}

By Lemma \ref{lem:devbound}, $-1\leq F^{(n+m)}_\omega(0)-F^{(m)}_{\sigma^n\omega}(0)-F^{(n)}_\omega(0)\leq 1$ for each $m,n\in\mathbb{N}$ and $\omega\in\Omega$. So
\begin{align}\label{eqn:Fnm-Fm-Fn}
& -1\leq  \int_\Omega F^{(n+m)}_\omega(0)\:\mathrm{d}\mathbb{P}- \int_\Omega F^{(m)}_{\sigma^n\omega}(0)\:\mathrm{d}\mathbb{P}-\int_\Omega F^{(n)}_\omega(0)\:\mathrm{d}\mathbb{P} \leq 1 .
\end{align}
Since the measure $\mathbb{P}$ is $\sigma$-invariant, by (\ref{eqn:Fnm-Fm-Fn}) we have $-1\leq R_{n+m}-R_m-R_n\leq 1$. For each $n\in\mathbb{N}$, define $R^+_n=R_n+1$ and $R^-_n=R_n-1$. So 
\begin{align}
(R^+_{n+m}-1)-(R^+_{m}-1)-(R^+_{n}-1) \leq 1 \implies R^+_{n+m}\leq R^+_n +R^+_m,
\end{align}
and thus the sequence $(R^+_n)$ is subadditive. Hence by Fekete's subadditive lemma, we have that $\lim_{n\to\infty} (1/n)R^+_n=R^+$ exists and for each $n\in\mathbb{N}$
\begin{align}\label{eqn:Feketebound}
\frac{1}{n}R^+_n \geq R^+.   
\end{align}\label{eqn:}
In a similar way, the sequence $(R^-_n)$ is superadditive, and so $\lim_{n\to\infty} (1/n)R^-_n=R^-$ exists and
$(1/n)R^-_n \leq R^-$ for each $n\in\mathbb{N}$. Since $R^+_n-1=R_n=R^-_n+1$, it follows that $R^-=R^+=R$ and $\lim_{n\to\infty}(1/n)R_n=R$. Thus by (\ref{eqn:Feketebound}), for each $n\in\mathbb{N}$, we have $R\leq (1/n)R^+_n=(1/n)R_n+1/n$, and similarly, $R\geq (1/n)R^-_n=(1/n)R_n-1/n$, and so
\begin{align}
\left|\frac{1}{n}R_n-R\right|\leq \frac{1}{n}.    
\end{align}

It remains to show that $R=\hat{\rho}(F)$. We have
\begin{align}
\left|\frac{1}{n}F^{(n)}(0)\right| 
 \leq \frac{1}{n}\sum_{i=0}^{n-1} \left|F_{\sigma^i\omega}(F^{(i)}_\omega(0))-F^{(i)}_\omega(0)\right|  
 \leq \frac{1}{n}\sum_{i=0}^{n-1} \| \Delta_{F_{\sigma^i\omega}} \|,
\end{align}
which converges as $n$ tends to infinity, by the Birkhoff ergodic theorem, to an integrable function $G\in L^1(\Omega)$ that satisfies
\begin{align}
\int_\Omega |G| \:\mathrm{d}\mathbb{P}=\int_\Omega \|\Delta_F\|\:\mathrm{d}\mathbb{P}<\infty.
\end{align}
Thus when $n$ is sufficiently large, the sequence $|(1/n)R_n|$ is dominated by the integrable function $G+1$. Hence by the dominated convergence theorem,
\begin{align*}
R & = \lim_{n\to\infty} \int_\Omega \frac{1}{n}F^{(n)}_\omega(0)\:\mathrm{d}\mathbb{P} 
 =  \int_\Omega \lim_{n\to\infty} \frac{1}{n}F^{(n)}_\omega(0)\:\mathrm{d}\mathbb{P} 
 = \int_\Omega \rho_F \:\mathrm{d}\mathbb{P} 
 = \hat{\rho}(F).
\end{align*}
\end{proof}

\section{Numerical Algorithms}\label{sec:NA}

In this section, we present three numerical algorithms for approximation of the random rotation number: a classical method based on Li and Lu's definition, as well as methods based on the lift-free approaches described in Section \ref{sec:liftfree}.

\subsection{Classical approach}
First we present an algorithm based on Li and Lu's definition of the random rotation number. This approximation method requires making a choice of lift  $F\in\tilde{\mathcal{H}}(\Omega)$ of a random circle homeomorphism. Using equation (\ref{eqn:randomrotationnumber}), we can approximate the random rotation number $\rho_F(\omega)$ for any $x_0\in\mathbb{R}$ and almost every $\omega\in\Omega$ by
\begin{align}\label{eq:classical}
A_F(n,\omega_0,x_0):= \frac{F^{(n)}_{\omega_0}(x_0)-x_0}{n},
\end{align}
for fixed $n\in\mathbb{N}$. 
Note however that the $1/n$ error bound obtained in the deterministic case does not hold in general.

It is enough to specify the values of the function $F_\omega$ on the interval $[0,1)$ for each $\omega\in\Omega$, since we can recover the values outside of this range by using the degree one property (\ref{eq:degreeone}): for $x\in\mathbb{R}$, we have $x=\pi(x)+\lfloor x \rfloor$ and so
\begin{align}
F_\omega(x) = F_\omega(\pi(x)+\lfloor x \rfloor) = F_\omega|_{[0,1)}(\pi(x))+\lfloor x \rfloor.
\end{align}
Pseudocode for the classical method is presented in Algorithm 1.

\begin{algorithm}[hbt]
\caption{Finite approximation of the random rotation number}
\label{alg1}
\begin{algorithmic}[1]
\Require{$F:(\omega,x)\mapsto F(\omega,x)$ (Restriction to $[0, 1)$ of a lift of $f$),
$\sigma:\omega\mapsto \sigma(\omega)$ (base dynamics), 
$\omega_0$ (initial base point), 
$x_0$ (initial fibre point), 
$n$ (number of iterations)}
\Ensure{$A_F(n,\omega_0,x_0)$}
\State $\omega \gets \omega_0$
\State $x \gets x_0$
\State $k \gets 0$ 
\For{$i = 1$ to $n$}
\State $k \gets k+\lfloor x\rfloor$ 
\State $x \gets F(\omega, x-\lfloor x \rfloor)$ 
\State $\omega \gets \sigma(\omega)$ 
\EndFor
\State \textbf{Return} $\frac{k + x-x_0}{n}$
\end{algorithmic}
\end{algorithm}

The computational complexity of Algorithm 1 is straightforward to quantify. Each iteration of the loop involves a single evaluation of the fibre map $F$, an update of the base state, and a fixed number of floor and fractional part operations. The total runtime scales linearly with the number of iterations $n$, yielding a time complexity of $O(n)$. Regarding memory usage, the algorithm only requires storage for the current state variables, which is independent of $n$. Thus, the space complexity of the method is $O(1)$.

\subsection{Binary coding approach}

The binary coding method can also be used to estimate the random rotation number. This method does not depend on a choice of lift, and approximates the random rotation number by the $n$th binary coding frequency $B_f(n,\omega_0,x_0)$, for fixed $n\in\mathbb{N}$.

The interval $J_\omega$ supporting the right branch is defined in terms of the point $c_f=(f_\omega)^{-1}(0)$. For an arbitrary function, finding the inverse of $0$ may not be numerically efficient. Therefore we use another approach to test whether or not a point lies in the interval $J_\omega$, without using the inverse map $(f_\omega)^{-1}$, that compares the value of $f_\omega(x)$ with the value of $f_\omega(0)$: namely $x\in J_\omega$ holds if $\hat{f}_\omega(x)<\hat{f}_\omega(0)$.

Pseudocode for the binary coding method is presented in Algorithm 2. A counter $k$ is used to record the number of visits to the right interval. In each iteration of the loop, the interval map $\hat{f}$ is evaluated once for $x$ and once for $0$ and a single comparison is performed to determine whether  $k$ should be increased, and then the base state is updated. Since these operations all have constant cost, the total runtime grows linearly, and so the time complexity is $O(n)$. The algorithm stores only the current fibre state $x$, base state $\omega$ and the counter $k$, and so the memory requirement is independent of $n$. Thus the space complexity if $O(1)$.

\begin{algorithm}[hbt]
\caption{Binary coding method to approximate the random rotation number}
\begin{algorithmic}[1]
\Require{$\hat{f}:(\omega,x)\mapsto \hat{f}(\omega,x)$ (Interval map equivalent to $f$),
$\sigma:\omega\mapsto \sigma(\omega)$ (base dynamics), 
$\omega_0$ (initial base point), 
$x_0$ (initial fibre point), 
$n$ (number of iterations)}
\Ensure{$B_f(n,\omega_0,x_0)$}
\State $\omega \gets \omega_0$
\State $x \gets x_0$
\State $k \gets 0$ 
\For{$i = 1$ to $n$}
\State $x \gets \hat{f}(\omega,x)$ 
\If{$x<\hat{f}(\omega,0)$}
\State {$k \gets k+1$} 
\EndIf
\State $\omega \gets \sigma(\omega)$ 
\EndFor
\State \textbf{Return} $\frac{k}{n}$
\end{algorithmic}
\end{algorithm}

\subsection{Visit counting approach}

The visit counting method is a third way to estimate the random rotation number. This method does not depend on a choice of lift, and approximates the random rotation number by the $n$th visit frequency $V_f(n,\omega_0,x_0,z)$, for fixed $n\in\mathbb{N}$.

To test whether a point lies in the fundamental domain, we need to consider two cases. If $\hat{f}_\omega(z)>z$, then the fundamental domain is a subinterval of $[0,1)$, whereas if $\hat{f}_\omega(z)\leq z$, then the fundamental domain consists of two disjoint subintervals $[0,\hat{f}_\omega(z))$ and $[z,1)$.

Pseudocode for the visit counting method is presented in Algorithm 3. The algorithm uses a counter $k$ to record how many times the interval condition is satisfied. In each iteration of the loop, the algorithm evaluates the map $\hat{f}$ to update the fibre point and then to determine the right endpoint $y$ of the fundamental domain, checks whether the current fibre point lies in the fundamental domain interval to decide whether the counter $k$ should be increased, and then updates the base point ready for the next iteration. All these operations take constant time, so the time complexity is $O(n)$. The algorithm only needs to store the values of $y$, $k$ and the current base and fibre points. The memory usage therefore remains constant, giving a space complexity of $O(1)$.

\begin{algorithm}[hbt]
\caption{Visit counting method to approximate the random rotation number}
\begin{algorithmic}[1]
\Require{$\hat{f}:(\omega,x)\mapsto \hat{f}(\omega,x)$ (Interval map equivalent to $f$),
$\sigma:\omega\mapsto \sigma(\omega)$ (base dynamics), 
$\omega_0$ (initial base point), 
$x_0$ (initial fibre point), $z$ (reference point),
$n$ (number of iterations)}
\Ensure{$V_f(n,\omega_0,x_0,z)$}
\State $\omega \gets \omega_0$
\State $x \gets x_0$
\State $k \gets 0$ 
\For{$i = 1$ to $n$}
\State $x \gets \hat{f}(\omega,x)$
\State {$y \gets \hat{f}(\omega,z)$}
\If{$z\leq y$}
\If{$z\leq x<y$}
\State $k \gets k+1$
\EndIf
\Else
\If{$z\leq x$ or $x<y$}
\State $k \gets k+1$
\EndIf
\EndIf
\State $\omega\gets \sigma(\omega)$
\EndFor
\State \textbf{Return} $\frac{k}{n}$ 
\end{algorithmic}
\end{algorithm}

\subsection{Comparison of methods}

Given a random circle homeomorphism $(\sigma,f)$, we have several methods to approximate the random rotation number over $n$ iterations. For each $x_0\in S^1$ and $\mathbb{P}$-almost every $\omega_0\in\Omega$, by Theorem \ref{thm:visitcounting}, the visit counting and binary coding methods give the same value:
\begin{align}
V_f(n,\omega_0,x_0,0)=B_f(n,\omega_0,x_0).
\end{align}
Moreover, by Theorem \ref{thm:binarycoding}, 
\begin{align}
\left| A_F(n,\omega_0,0) -B_f(n,\omega,0) \right| <\frac{1}{n},
\end{align}
so these approximations are within $1/n$ of the classical method with the standard lift $F$ of $f$. 
Hence the three approaches to estimating the random rotation number $\rho_F(\omega_0)$ along a single trajectory all produce very similar outputs.

The visit counting method and binary coding method have two advantages of over the classical method. Firstly, the independence of the choice of lift. Secondly, the visit counting method and binary coding method both rely on counting how often certain events occur and so are integer-based methods, and so are less prone to rounding error. On the other hand, as the visit counting and binary coding algorithms require additional steps in each iteration of the loop, their runtimes are proportionally higher than for the classical algorithm.

\section{Numerical results}\label{Sec:numerics}

In this section we use the numerical techniques to estimate the random rotation number for examples of random circle homeomorphisms. As described in Section \ref{sec:errorbounds}, we can get estimates of the random rotation number with known error bounds by integration over the base space $\Omega$. By Theorem \ref{thm:1N}, we have
\begin{align}\label{eqn:intAF}
\left|\int_\Omega A_F(n,\omega,x_0)\:\mathrm{d}\mathbb{P}-\hat{\rho}(F)\right|\leq\frac{1}{n}.
\end{align}

In order to get an estimate of the mean random rotation number $\hat{\rho}(F)$, We can use any one of a variety of numerical integration methods to approximate the integral, for example quadrature rules or Monte-Carlo integration. Each method introduces some numerical error that depends on the regularity properties of the integrand.

In order to reduce some of the issues related to numerical integration, we test the algorithms on examples where the base space $\Omega$ is the unit interval $[0,1)$ or circle $S^1$ and the invariant probability measure $\mathbb{P}$ is Lebesgue measure.

Given $m\in\mathbb{N}$, consider the uniform partition
\begin{align}
P_m=\left\{\omega_j=\frac{j}{m}: j=0,\ldots, m\right\}.
\end{align}
In order to approximate the integral in (\ref{eqn:intAF}), we take a Riemann sum of the form
\begin{align}\label{eqn:mean}
R_F(n,m,x_0)=\frac{1}{m}\sum_{j=1}^m A_F(n,\omega_j,x_0).
\end{align}
In other words, the integral is approximated by averaging the value of $A_F(n,\omega,x_0)$ across all $\omega$ in the uniform partition $P_m$.

To test the convergence of the numerical algorithms, we start with an example for which the value of the random rotation number can be determined precisely.

\begin{ex}\label{ex:convergence}
For the base dynamics, we take $\Omega=S^1$ and define $\sigma:\Omega\to\Omega$ to be the golden rotation $\sigma(\omega)=\omega+(\sqrt{5}-1)/2\:\mathrm{mod}\:1$, so the Lebesgue measure $\mathbb{P}$ on $\Omega$ is preserved.
For the fibre map we take the lift $f\in\mathcal{H}(\Omega)$ of a random Arnold homeomorphism, as defined in (\ref{eq:Arnold}), with 
\begin{align}
\alpha(\omega)=\sin(2\pi\omega)
\quad\mathrm{and}\quad
\beta(\omega)=\begin{cases}
+1 & \ \ \:0\leq \omega <1/2 \\
\ 0 & 1/2\leq \omega<3/4 \\
-1 & 3/4\leq \omega<1
\end{cases}
\end{align}
and so $F_\omega$ is a piecewise continuous function of $\omega$. Note that each fibre map acts as a translation when restricted to the set of the integers: namely, $F_\omega(k)=k+\beta(\omega)\in\mathbb{Z}$ for each $k\in\mathbb{Z}$ and $\omega\in\Omega$.

Therefore $\Delta_{F_\omega}(k)=\Delta_{F_\omega}(0)=\beta(\omega)\in\mathbb{Z}$ for each $k\in\mathbb{Z}$ and $\omega\in\Omega$, and so we have
\begin{align*}
\frac{1}{n}F^{(n)}_\omega(0) 
 = \frac{1}{n}\sum_{i=0}^{n-1} \Delta_{F_{\sigma^i\omega}}(F^{(i)}_\omega(0)) 
 = \frac{1}{n}\sum_{i=0}^{n-1} \Delta_{F_{\sigma^i\omega}}(0) 
 = \frac{1}{n}\sum_{i=0}^{n-1} \beta(\sigma^i\omega) 
\end{align*}
which converges by the Birkhoff ergodic theorem for $\mathbb{P}$-almost every $\omega\in\Omega$ to
\begin{align}
\int_\Omega \beta\:\mathrm{d}\mathbb{P} = \frac{1}{2}\times (+1) + \frac{1}{4}\times 0 +\frac{1}{4}\times(-1) = \frac{1}{4}\,.
\end{align}
Hence the precise value of the mean random rotation number $\hat{\rho}(F)$ is $1/4$. Since the base map is ergodic, it follows that $\rho_F(\omega)=1/4$ for $\mathbb{P}$-almost every $\omega\in\Omega$.

The random rotation number is independent of the choice of $x_0$, so for testing the numerical algorithms, we take a non-integer initial point $x_0=0.3$ so as not to take special advantage of the fact that the fibre maps act as translations on the integers. Using the classical approach of Algorithm 1, the approximations $A_F(n,\omega_0,0.3)$ were calculated for $n=1,\ldots,1000$. Figure \ref{fig:convergence} shows the approximations for the single trajectory (using $\omega_0=0$) and the averaged approximation $R_F(n,100,0.3)$ using the uniform partition $P_{100}$. The averaged approximation $R_F(n,100,0.3)$ lies well within the $1/n$ bounds, although the single trajectory approximation $A_F(n,0,0.3)$ does not.

\begin{figure}[htb]
    \centering
    \includegraphics[scale=1.00]{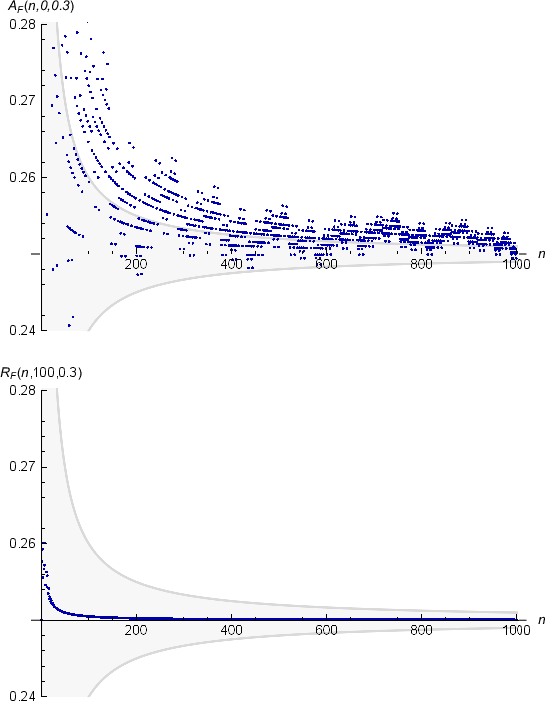}  
    \caption{Graph showing approximations of the random rotation number against $n$ based on a single trajectory (top) and by averaging over 100 trajectories (bottom), along with the error bounds $1/4\pm1/n$ (light grey).}
    \label{fig:convergence}
\end{figure}
\end{ex}

Next we consider an example with a non-continuous base transformation and for which the dependence of the fibre maps on the base is also non-continuous.

\begin{ex}\label{ex:IETbase}
For the base dynamics, we take $\Omega=[0,1)$ and define $\sigma:\Omega\to\Omega$ to be the interval exchange transformation given by
\begin{align}\label{eqn:iet}
\sigma(\omega) =
\begin{cases}
\omega + 1 - u, & \text{if }\ 0 \leq \omega < u \\
\omega + 1 - u - v, & \text{if }\ u \leq \omega < v \\
\omega - v, & \text{if }\ v \leq \omega < 1 
\end{cases}
\end{align}
where $u=\sqrt{3}/3$ and $v=\sqrt{2}/2$. A plot of $\sigma$ is shown in Figure \ref{fig:3iet}. The map preserves Lebesgue measure $\mathbb{P}$, and since $1$, $u$ and $v$ are not rationally related, $\mathbb{P}$ is in fact the unique invariant probability measure for $\sigma$ by Keane's condition \cite{Keane1975}. This choice of base map is not within the quasiperiodic setting, since $\sigma$ not a rotation nor is measure-theoretically isomorphic to a rotation (see Dajani \cite{Dajani2002}).

\begin{figure}[htb]
    \centering
    \includegraphics[scale=1.00]{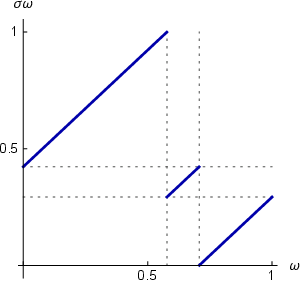}  
    \caption{Graph of the interval exchange transformation $\sigma$ defined in (\ref{eqn:iet}).}
    \label{fig:3iet}
\end{figure}

For the fibre map we take a random Arnold homeomorphism  $f\in\mathcal{H}(\Omega)$, as defined in (\ref{eq:Arnold}), with 
\begin{align}
\alpha(\omega)=\frac{\sin(2\pi \omega)+\{5\omega^2\}}{2} \quad\textrm{and}\quad \beta(\omega) = \frac{1-3\omega^2+\{2\sin(2\pi \omega)\}}{2},
\end{align}
where $\{x\}=x-\lfloor x\rfloor\in[0,1)$ denotes the fractional part of $x\in\mathbb{R}$. Thus $\alpha$ and $\beta$ are piecewise-continuous functions of $\omega$ (see Figure \ref{fig:fibremap}). 

\begin{figure}[htb]
    \centering
    \includegraphics[scale=1.00]{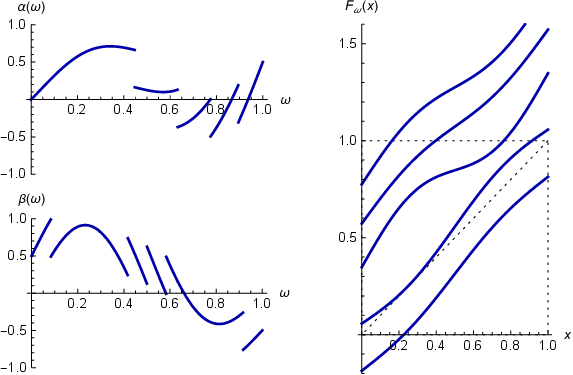}  
    \caption{Graphs of $\alpha(\omega)$ and $\beta(\omega)$ against $\omega\in[0,1)$ (left) and $F_\omega(x)$ against $x\in[0,1)$ for selected $\omega$ (right).}
    \label{fig:fibremap}
\end{figure}

Using the uniform partition $P_{100}$ of $100$ points, the approximations $R_F(n,100,0)$ of the random rotation number were calculated (see Figure \ref{fig:1nconvergence}). The approximations $R_F(n,100,0)$ appear well within the $1/n$ bounds. 

\begin{figure}[htb]
    \centering
    \includegraphics[scale=1.00]{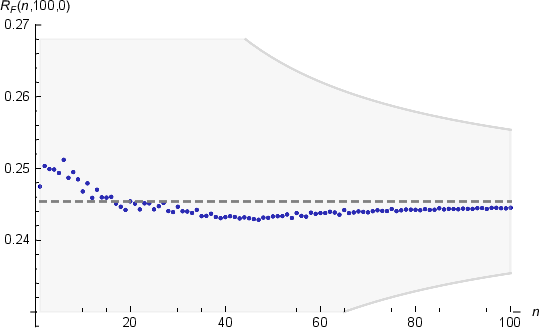}  
    \caption{Graph showing approximations $R_F(n,100,0)$ of the random rotation number against $n$, along with the estimated error bounds $R_F(1000,100,0)\pm1/n$ (light grey).}
    \label{fig:1nconvergence}
\end{figure}
\end{ex}

Finally, as an application of the numerical algorithms, we investigate the mean rotation number $\hat{\rho}(F_a)$ along a one-parameter family of random circle homeomorphisms $\{F_a\}_{a\in [0,1]}$, to observe the dependence of $\hat{\rho}(F_a)$ on $a$.

\begin{ex}\label{ex:1parameterfamily}
For the base dynamics, we take $\Omega=[0,1)$ and the same interval exchange transformation as in the previous example.

For the fibre map we take a random Arnold homeomorphism (see (\ref{eq:Arnold})) $f\in\mathcal{H}(\Omega)$ with 
\begin{align}
\alpha(\omega)=\frac{9+\{\sqrt{2}\,\omega\}}{10} \quad\textrm{and}\quad \beta(\omega) = \frac{\{\pi \omega\}}{5}.
\end{align}
Thus the standard lift $F$ of $f$ is given by
\begin{align}
F_\omega(x)= x+\frac{9+\{\sqrt{2}\,\omega\}}{20\pi}\sin(2\pi x)+\frac{\{\pi \omega\}}{5}.
\end{align}

Note that the dependence of $F_\omega$ on $\omega$ is measurable but not continuous. For some values of $\omega$ such as $0$, the map $F_\omega$ has fixed points, while for other values, such as $\omega=0.3$, there are no fixed points.
Since $\sigma$ is ergodic, the random rotation number is $\mathbb{P}$-almost everywhere equal to $\rho(F)$.

Consider now the one-parameter family $(F_a)_{a\in [0,1]}$, where for each $a\in [0,1]$, $F_a\in\tilde{\mathcal{H}}(\Omega)$ is given by $F_a=F+a$. As proved by Li and Lu, the mapping $a\mapsto \hat{\rho}(F_a)$ is a continuous function $[0,1]\to\mathbb{R}$. Figure \ref{fig:3ietstaircase} shows the dependence of the approximation $R_{F_a}(500,100,0)$ of $\hat{\rho}(F_a)$ on the parameter $a$. Note the evidence of `phase-locking', with the graph locally constant on certain intervals. 

\begin{figure}[htb]
    \centering
    \includegraphics[scale=1.00]{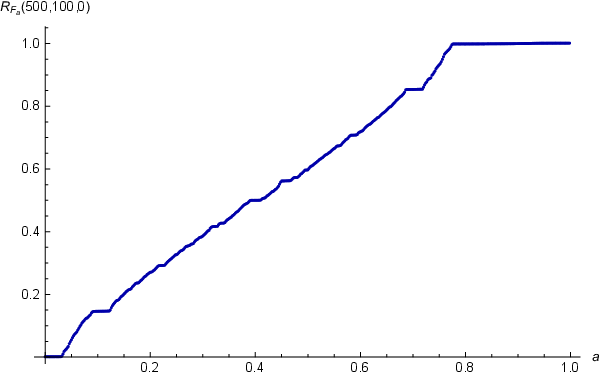}  
    \caption{Graph showing the dependence on $a$ of the approximation $R_{F_a}(500,100,0)$ of $\hat{\rho}(F_a)$.}
    \label{fig:3ietstaircase}
\end{figure}
\end{ex}

\section{Conclusions}

In this paper, we showed that the random rotation number of a random circle homeomorphism can be defined without reference to any lifts, using binary coding (\ref{eq:binary}) or using visit counting (\ref{eq:visit}), each giving the same value as when combining the classical definition (\ref{eqn:randomrotationnumber}) with the standard lift (\ref{eq:standardlift}). We used each of these definitions to set out algorithms for approximating the random rotation number.

We proved that the mean value of the random rotation number can be approximated by the integral (\ref{eq:intapprox}) over the base space of the $n$th approximation to within an error of at most $1/n$. Motivated by this result, we performed numerical computations on random Arnold homeomorphisms, where the dependence on the noise parameter is not continuous. The calculations indicated that improved estimates of the mean random rotation number can be achieved by averaging over multiple realisations of the noise. In Examples \ref{ex:convergence} and \ref{ex:IETbase}, the averaged approximation $R_F$ can be observed (in Figures \ref{fig:convergence} and \ref{fig:1nconvergence} respectively) to lie well within the bounds arising from Theorem \ref{thm:1N}. So we leave as an open problem the question of under which circumstances (if any) the bounds can be sharpened. 


\end{document}